\documentclass[11pt]{amsart}
\usepackage{amssymb, amsmath, amsthm}
\usepackage[latin1]{inputenc}
\usepackage{graphicx}
\usepackage[final]{hyperref}
\usepackage{color}
\usepackage{ulem}
\usepackage{epstopdf}

\usepackage[a4paper, centering]{geometry}
\usepackage{color}
\usepackage{graphicx}
\usepackage{scalerel,stackengine}

\usepackage{mathtools}

\newtheorem{theorem}{Theorem}
\newtheorem{proposition}[theorem]{Proposition}
\newtheorem{lemma}[theorem]{Lemma}
\newtheorem{corollary}[theorem]{Corollary}
\theoremstyle{definition}

\definecolor{verde}{RGB}{20,150,100}

\newcommand{\EEE}{\color{black}}

\newcommand{\Om}{\Omega}

\newcommand{\R}{\mathbb R}
\newcommand{\N}{\mathbb N}

\stackMath
\newcommand\reallywidecheck[1]{%
\savestack{\tmpbox}{\stretchto{%
  \scaleto{%
    \scalerel*[\widthof{\ensuremath{#1}}]{\kern-.6pt\bigwedge\kern-.6pt}%
    {\rule[-\textheight/2]{1ex}{\textheight}}
  }{\textheight}%
}{0.5ex}}%
\stackon[1pt]{#1}{\scalebox{-1}{\tmpbox}}%
}

\def\R{\mathbb{R}}

\def\N{\mathbb{N}}

\def \d{\delta}
\def \e{\varepsilon}

\def \res{\mathop{\hbox{\vrule height 7pt width .5pt depth 0pt
\vrule height .5pt width 6pt depth 0pt}}\nolimits}

\def\1{{{\bf 1}}}

\begin{document}    
\title[]{Optimal partitions \\ for Robin Laplacian eigenvalues}
\author[]{Dorin Bucur, Ilaria Fragal\`a, Alessandro Giacomini}
\thanks{D.B. is member of the Institut Universitaire de France. I.F. and A.G. are members of the Gruppo Nazionale per L'Analisi Matematica, la Probabilit\`a e loro Applicazioni (GNAMPA) of the Istituto Nazionale di Alta Matematica (INdAM)
}

\address[Dorin Bucur]{
Univ. Grenoble Alpes, Univ. Savoie Mont Blanc, CNRS, LAMA\\
73000 Chambéry (France)
}
\email{dorin.bucur@univ-savoie.fr}

\address[Ilaria Fragal\`a]{
Dipartimento di Matematica \\ Politecnico  di Milano \\
Piazza Leonardo da Vinci, 32 \\
20133 Milano (Italy)
}
\email{ilaria.fragala@polimi.it}

\address[Alessandro Giacomini]{
DICATAM, Sezione Matematica \\ 
Universit\`a degli Studi di Brescia\\
Via Branze 43\\
25123 Brescia (Italy) 
}
\email{alessandro.giacomini@unibs.it}


\date{\today}

\begin{abstract}
{We prove the existence of an optimal partition for the multiphase shape optimization problem which consists in minimizing the sum of the first Robin Laplacian eigenvalue of $k$  mutually disjoint {\it open} sets which have a $\mathcal H ^ {d-1}$-countably rectifiable boundary and are contained into a given box $D$ in $\R ^d$. 
 }
   \end{abstract}
\maketitle

\section{Introduction}

 Aim of this paper is to study a multiphase shape optimization problems with Robin boundary conditions. More precisely, given an open bounded subset $D$ of $\R ^ d$ with Lipschitz boundary, we consider the optimal partition problem
 $$(P) \qquad \inf \left\{ \sum _{i=1} ^k \lambda _ 1 (\Omega _i, \beta) \ :\  (\Omega _1, \dots, \Omega_k) \in \mathcal A ( D)\right\} \,, $$ 
 where $\mathcal A (D)$ is the class of $k$-tuples of open domains contained into $D$ 
 such that $\Omega _i \cap \Omega _j=\emptyset$ for $i \neq j$, and 
 $\lambda _1 (\Omega, \beta)$ denotes the 
first Robin Laplacian eigenvalue, for a fixed parameter $\beta >0$.  If $\Omega$ is sufficiently smooth (say Lipschitz), it is defined by
 \begin{equation}\label{f:defrobin}\lambda _ 1 (\Omega, \beta):=\inf_  {u \in H ^ 1 (\Omega)  \setminus \{0\} }   \frac{ \int_{\Omega} |\nabla u | ^ 2  \, dx   + \beta \int_{\partial \Omega} u ^2\, d \mathcal H ^ { d-1} }
  { \int _{\Omega} u^ 2 \, dx }     \, . 
  \end{equation} 

The analogous to problem $(P)$ in which $\lambda _ 1 (\Omega, \beta)$ is replaced by the first Dirichlet eigenvalue $\lambda _1 (\Omega)$ of the Laplacian has been extensively studied in the last years. In that case, the starting point of the analysis was the existence result in the class of quasi-open sets given in \cite{BuBuHe},  and then the main existence statements in the class of 
open sets were proved by Conti-Terracini-Verzini \cite{CTV03, CTV05bis, CTV05} and Caffarelli-Lin \cite{CL07}; meanwhile and subsequently, regularity results for optimal partitions have been obtained in \cite{CL07, CTV03, HeHoTe, Ramos};  without attempt of completeness, further related works are \cite{BoVe, BNHV, BV14, He07, He10}. 

In contrast, problem (P) that we view as a prototype of a multiphase free discontinuity problem, seems to be completely unexplored. 
The topic which drove 
our attention to problem (P) was its connection with optimal Cheeger partitions. In fact, by combining the results about the asymptotics of optimal Cheeger clusters recently obtained in \cite{BF17, bfvv17} with the relationship between $\lambda _ 1 (\Omega, \beta)$ and the quotient $\beta \frac{|\partial \Omega|}{|\Omega|}$,  the first and second authors have proved in \cite{bf17R} that, in dimension $d =2$ and when the number of cells $k$ becomes very large, optimal partitions for problem $(P)$ form a very special ``hexagonal'' pattern.  Namely, if 
 $r _k (D, \beta)$ denotes the infimum of problem $(P)$, when the cells $\Om_i$ are convex planar sets, it holds
$$
\lim _{k\to + \infty} \frac{|D|^ {1/2} }{ k ^ {3/2} } r _k (D, \beta) 
=  \beta h ( H)
\, , 
$$
\smallskip
being $h ( H)$ the Cheeger constant of the unit area regular hexagon (see \cite{Leo} for an overview about the Cheeger problem). Remarkably, an analogous honeycomb-like asymptotical behaviour of optimal partitions in the Dirichlet case, conjectured by Caffarelli and Lin in \cite{CL07}, is still unproved. 
\par
Being this the state of the art, we were led in a natural way to investigate existence and regularity issues for multiphase problems 
in the Robin case.  
As a first contribution in this direction, the present work concerns the existence of optimal partitions in the class of open sets. 
Due to the presence of the  boundary term in \eqref{f:defrobin}, the shape functional $\lambda _ 1(\Omega, \beta)$ behaves quite differently from $\lambda _ 1(\Omega)$ (in particular, it  lacks monotonicity under domain inclusion), and obtaining the existence of solutions requires a completely different approach.  
Our strategy moves along the way traced by the first and third authors in some previous works about variational problems for the Robin Laplacian (see  \cite{BG10, BG15, BG16}). In particular, we settle our partition problem in the 
class 
\begin{equation}\label{f:defA} \begin{array}{ll}
\mathcal A (D)  := \Big \{& (\Omega _1, \dots, \Omega _k) \ :\   \Omega _i \subset D \, ,  \Omega _i \cap \Omega _j=\emptyset\text{ for } i \neq j 
\\ \noalign{\smallskip} 
& \Omega _i \text{ open},\  \partial \Omega_i \  \mathcal H ^ {d-1} \text{-countably rectifiable with}\  \mathcal H ^ {d-1} (\partial \Omega_i) < +\infty  \Big \} \,. \end{array} 
\end{equation}
 Notice carefully that the sets $\Omega _i$ do not need to be Lipschitz, and in particular they may contain inner cracks. 
Nevertheless,  the definition of first Robin eigenvalue can be extended to any  open set $\Omega$ having a
$\mathcal H ^ {d-1}$-countably rectifiable boundary of finite $\mathcal H ^ {d-1}$ measure
by setting
  \begin{equation}\label{f:deflambda}
 \lambda _ 1 (\Omega , \beta):=  \inf _  {u \in (H ^ 1 (\Omega) \cap L ^ \infty (\Omega) ) \setminus \{0\} }   \frac{ \int_{\Omega} |\nabla u | ^ 2  \, dx   + \beta \int_{\partial \Omega} \big [(u^+) ^ 2 + (u ^-) ^ 2  \big ] \, d \mathcal H ^ { d-1} }
 { \int _{\Omega} u^ 2 \, dx }     \, .  
\end{equation}
Here $u ^ \pm$ denote the two traces of the $BV$ function $u$ (extended to zero outside $\Omega)$ along $\partial \Omega$ (see Section \ref{sec:prel} for more details). In particular, if $\Omega$ is Lipschitz, one of the traces is $0$ and the other one coincides with the usual trace of $u$ in $H ^ 1 (\Omega)$, so that definition \eqref{f:deflambda} gives back  \eqref{f:defrobin}. 

In this setting we obtain the following existence result:

 \begin{theorem}\label{t:robin}
 Let the class $\mathcal A(D)$ be given by \eqref{f:defA} and, for every $(\Omega_1, \dots, \Omega _k) \in \mathcal A (D)$, let $\lambda _1 (\Omega_i,\beta)$ be defined by \eqref{f:deflambda}. Then problem $(P)$ admits a solution. 
\end{theorem}
 
The idea of our proof follows the pioneering point of view introduced by  Alt and Caffarelli  in \cite{AlCa} in order to deal with a one phase free boundary problem with Dirichlet boundary conditions. In fact, we pass through a relaxed formulation of the problem which is a minimization on functions rather than on sets. In view of the weak definition of $\lambda _ 1 (\Omega, \beta)$ given in \eqref{f:deflambda}, 
the space $SBV(\R ^d)$  of special functions of bounded variation introduced by De Giorgi and Ambrosio in \cite{DGA} appears as the ideal ambient to study our free discontinuity problem. More precisely, taking into account that we have to model multiple phases, 
we need to consider the class of vector fields
 $$\mathcal F (D)  := \Big \{( u _1, \dots, u _k )  \in (SBV ^ {\frac{1}{2}}(\R ^d)) ^ k \ :\ {\rm supp} (u_i) \subseteq \overline D \, , \ u _i \geq 0 \, , \ u _i \cdot u _j = 0 \text{ in } D \Big \},$$ 
where 
$SBV ^ {\frac{1}{2}} (\R^d)$ denotes the space of nonnegative functions $u \in L ^ 2 (\R^d)$ such that $u ^ 2$ is in $SBV(\R ^d)$. 
\par
Our relaxed functional form of problem $(P)$ reads
$$
(\overline P ) 
 \quad \inf \left\{ \sum _{i=1} ^k \frac{ \int_{\R ^d} |\nabla u_i | ^ 2  \, dx+ \beta \int_{J _{u _i}} \big [(u_i ^+) ^ 2 + (u_i ^-) ^ 2  \big ] \, d \mathcal H ^ { d-1} }
 { \int _{\R^d} u_i ^ 2 \, dx }   \ :\  ( u _1, \dots, u _k ) \in \mathcal F (D)  \right\}. 
$$

The proof of Theorem \ref{t:robin} starts with an existence result for problem $(\overline P)$ and is
carried over in Section \ref{sec:proof}.  Its main steps are highlighted below.  For the sake of clarity, we have collected in Section \ref{sec:prel} some background material which should allow  a self-contained comprehension of the statement as well as of the proof's outline.  

\begin{itemize}
\item[I.] {\it Existence in $\mathcal F (D) $}. Problem $(\overline P)$ admits a solution. This  follows immediately from 
the compactness 
and lower semicontinuity properties in $SBV^ {\frac{1}{2}}  (\R ^d)$
proved in \cite[Theorem 3.3]{BG10} for each  addendum of  the energy in $(\overline P)$. 
 \medskip 
\item[II.]  {\it Upper and lower bounds on the supports}.  If $(u _1, \dots, u _k) \in  \mathcal F (D)$ is a solution to problem $(\overline P)$  and $\omega _i$ denotes the support of $u _i$, for every $i = 1, \dots, k$ there exist constants $M _i, \alpha _i >0$ such that $M _i \geq u _i \geq \alpha _i$ a.e.\ on  $\omega _i$. 
This follows  from the optimality of $u$, by taking respectively the competitor $v _i = u _i \1 _{\{u \geq \e \}}$ to get the lower bound,  and $v_i  = u _i \wedge M$ to get the upper bound (leaving unchanged the phases $u _j$ for $j \neq i$). In particular, for the lower bound we need to settle a careful energy estimate in order to make work in our framework an iteration scheme which can be traced back in \cite{BL14} (and has been refined in \cite{BG16}).    \medskip
\item[III.] {\it SBV regularity and finite $\mathcal H ^ {d-1}$ measure of the jump sets}. For every $i = 1, \dots, k$, the function $u _i$ satisfies $\mathcal H ^ {d-1} ( J _{u _i} ) < + \infty$ and belongs to $SBV  (\R ^d)$. This is obtained as a quite direct consequence of the previous step. 
\medskip 
\item[IV.]  {\it Essential closedness of the jump sets}. For every $i = 1, \dots, k$, the set $J _{u_i} $ is essentially closed in $D$.Also the proof of this crucial regularity property exploits the bounds obtained Step II, but it is much more delicate. It relies on a uniform density estimate from below for the supports of $u _i$ (Lemma \ref{p:density}), which in turn is obtained by applying the Faber-Krahn inequality for the first Robin Laplacian eigenvalue established in \cite{BG15}. With the aid of the local isoperimetric inequality, such a density lower bound provides some regularity properties for the supports, and gives information concerning their interaction (Corollary \ref{cor:phases}). The closure property follows by combining these properties with the fact that $u_i$ is an {\it almost-quasi minimizer} for the Mumford-Shah functional  ``well inside'' its support.
\medskip  
\item[V.] {\it Identification of an optimal $k$-tuple in $\mathcal A (D)$ and conclusion.} Denoting by $\Omega _i$ the connected component of $\R^d \setminus \overline { J _{u _i}}$ where $u _i$ does not vanish, thanks to Step IV it turns out that $(\Omega _1, \dots, \Omega _k)$ belongs to $\mathcal A (D)$ and  solves problem $(P)$. 
\end{itemize}
 
 \smallskip
A natural open question which stems from Theorem \ref{t:robin} is to establish the regularity of the free boundaries in an optimal $k$-tuple.  As a first step in this direction, we show that the jump sets are Ahlfors regular in $D$ (see Proposition \ref{p:Ahlf}). 
For a one phase problem under Robin boundary conditions, some regularity properties have been recently obtained in \cite{CK16, K15}; we think it would be interesting to investigate their validity in our multiphase context.

To conclude, let us mention that our results can be extended, with minor modifications in the proofs, to other shape functionals under Robin boundary conditions, such as for instance thermal insulation of multiple obstacles, torsional rigidity or $p$-Laplacian energies. A more challenging target is to deal with optimal  partitions for higher eigenvalues of the Robin Lapacian.

  \section {Preliminaries} \label{sec:prel}

\subsection{Basic notation}  
\label{subsec:notation}
Throughout the paper $B_\rho(x)$ will denote the open ball with center $x\in \R^d$ and radius 
$\rho>0$. Given $E\subseteq \R^d$, $\1_E$ will stand for its characteristic function, while $|E|$ will denote its Lebesgue measure. In particular we set $\omega_d:=|B_1(0)|$. Finally, given $a,b\in \R$, we set $a\wedge b:=\min\{a,b\}$.

\subsection{Definition of $\lambda _1 (\Omega, \beta)$ for general open sets}
Let us recall how the notion of first  Robin Laplacian eigenvalue can be extended to domains with possibly irregular boundary. 
Assume that $\Omega$ is an open set with a 
$\mathcal H ^ {d-1}$-countably rectifiable boundary of finite $\mathcal H ^ {d-1}$ measure.  For
  $\mathcal H ^ {d-1}$-a.e.\ $x \in \partial \Omega$, 
  let $\nu (x) $ denote the unit outer normal to $\partial \Omega$ and   set $B _\rho ^ \pm (x, \nu (x)) := \{ y \in B _\rho (x) \ :\ (y-x) \cdot \nu (x) 
    \gtrless 0 \}$. 
Given $ u \in H ^ 1 (\Omega) \cap L ^ \infty (\Omega)$, extend it to zero outside $\Omega$, and let $u ^ \pm$ denote its two traces along $\partial \Omega$, defined by 
  \begin{equation}\label{f:traces}
  u ^ {\pm} (x):= \lim _{\rho \to 0 } \frac{1}{|B _\rho ^ \pm (x, \nu (x)) |}  \int _{ B _\rho ^ \pm (x, \nu (x))  } u ( y) \, dy \,.
  \end{equation}
  These limits turn out to 
  exist at $\mathcal H ^ {d-1}$-a.e.\  $x \in \partial \Omega$ because $u \1 _\Omega$ belongs to $BV (\R ^d)$ ({\it cf.} \cite[Proposition 4.4]{AFP}).  
   Then, for every  $\beta >0$ ({\it assumption which will be kept throughout the paper with no further mention}), following \cite[(3.4)]{BG15}, we set 
$$ \lambda _ 1 (\Omega , \beta):=  \min _  {u \in (H ^ 1 (\Omega) \cap L ^ \infty (\Omega) ) \setminus \{0\} }   \frac{ \int_{\Omega} |\nabla u | ^ 2  \, dx   + \beta \int_{\partial \Omega} \big [(u^+) ^ 2 + (u ^-) ^ 2  \big ] \, d \mathcal H ^ { d-1} }
 { \int _{\Omega} u^ 2 \, dx }     \, .  
$$

We point out that,  if $\Omega$ is Lipschitz, one of the traces of $u$ is $0$ and the other one coincides with the usual trace  in $H ^ 1 (\Omega)$, so that the above definition gives back  \eqref{f:defrobin}. 
  
  \smallskip
  
\subsection{The space $SBV ^ {\frac{1}{2}} ( \R ^d)$}  In the functional formulation of minimization problems for the above defined Robin eigenvalue,  the following space intervenes in a natural way (see \cite[Definition 3.1]{BG10}):
 $$
 SBV ^ {\frac{1}{2}} (\R ^d):= \Big \{ u \in L ^ 2 (\R ^d) \ :\ u \geq 0 \text{ a.e. in } \R ^d \, , \ u ^ 2 \in SBV (\R ^d) \Big \}\,.
 $$ 
  Here $SBV (\R ^d)$ denotes the space of {\it special functions of bounded variation}, namely functions $u \in L ^ 1 ( \R ^d)$ which have bounded variation  (meaning that the distributional gradient $Du$ is a measure) and are such that the singular part $D ^ s u$ of the measure $Du$ is concentrated on the {\it jump set} $J _u$. By definition, $J _u$ is the set of points $x \in \R ^d$ such that the approximate upper and lower limits $u ^ \pm (x)$ do not coincide:
\begin{equation}\label{f:aplimits} 
u ^ + (x) := \inf \big \{ t \in \R \ :\ x \in \{ u >t\} ^0 \big \}\, , \qquad  u ^ - (x) := \sup \big \{ t \in \R \ :\ x \in \{ u <t\} ^0 \big \},
\end{equation}  
where $E ^{(0)}$ stands for the set of points $y \in \R ^d$ at which $E$ has a zero density, {\it i.e.} 
$$
\lim _{\rho \to 0} \frac{|E \cap B _\rho (y) | }{\rho ^ d}= 0.
$$ 
Let us recall that every function $u \in SBV (\R ^d)$ is approximately differentiable a.e. (with approximate gradient denoted by $\nabla u$),  its jump set $J _u$ is $\mathcal H ^ {d-1}$-countably rectifiable, and the distributional gradient $Du$ is given by
$$Du (E) = \int _E \nabla u \, dx + \int_{J _u \cap E} ( u ^+- u ^-) \, \nu _u  \, d \mathcal H ^ {d-1}\qquad \text{for every Borel set } E\, . $$ 
Here $\nu _u: J _u \to \mathbb S ^ {d-1}$ is a Borel unit normal vector field such that, for $\mathcal H ^ {d-1}$-a.e.\ $x \in J _u$, the approximate upper and lower limits defined by \eqref{f:aplimits} agree with the traces defined by \eqref{f:traces} (taking $\nu = \nu _u$). 
\par
A case of special relevance is when $u$ is the characteristic function of a set $E$ of {\it finite perimeter}: in that case, the measure $Du$ is purely singular, and the jump set agrees with the {\it essential boundary} $\partial ^ e E$, defined by 
 $$
 \partial ^ e E:= \R ^ d \setminus \big (E ^ {(1)} \cup E^{(0)}  \big ),
 $$ 
where $E ^{(1)}$ stands for the set of points $y \in \R ^d$ at which $E$ has a density one.
 
 \medskip
 We refer the reader to the monograph \cite{AFP} for the theory of functions of bounded variation, and to \cite{BG10} for fine properties of functions in $SBV ^ {\frac{1}{2}} ( \R ^d)$. 
  
Here, in order to make sense of the relaxed energy in $(\overline P)$,  we limit ourselves to recall that a function $u \in SBV ^ {\frac{1}{2}} (\R ^d)$ is approximately differentiable a.e. (with approximate gradient still denoted by $\nabla u$)   
  and that its jump set (still denoted by $J _u$) is $\mathcal H ^ {d-1}$-countably rectifiable; moreover, $\nabla u$ and $J _u$ (the latter endowed with a unit normal vector $\nu _u$) are related respectively to the absolutely continuous and to the singular part of $D(u ^ 2)$ by the identities
$$
(\nabla u ^ 2) dx = (2 u \nabla u ) dx \, , \qquad\text{and}\qquad D ^ s (u ^ 2)  = [(u^+) ^ 2 - ( u ^ -) ^ 2 ] \nu _u \mathcal H ^ {d-1} \res J _u\,.
$$

\smallskip
\subsection{Faber-Krahn inequalities}  

The main results of \cite{BG15} can be summarized in the equalities
\begin{equation}
\label{f:FK}
\min_{\stackrel{u\in SBV^{1/2}(\R^d),}{|supp(u)|=m}}
\frac{\int_{\R^d}|\nabla u|^2\,dx+\int_{J_u}[(u^+)^2+(u^-)^2]\,d\mathcal H^{d-1}}{\int_{\R^d}u^2\,dx}
=
\min_{|\Om|=m}\lambda_1(\Om,\beta)=\lambda_1(B,\beta),
\end{equation}
where $m>0$, $B$ is a ball such that $|B|=m$, and $\Om$ varies in the class of open sets with a 
$\mathcal H ^ {d-1}$-countably rectifiable boundary of finite $\mathcal H ^ {d-1}$ measure. Minimizers of the functional problem are supported on balls, the function coinciding with the associated first eigenfunction of the Robin-Laplacian. Finally, open sets optimal for $\lambda_1$ coincide with a ball up to $\mathcal H^{d-1}$-negligible sets.

 \smallskip
\subsection{Almost-quasi minimizers for the Mumford-Shah functional}
\label{sec:mon}
We will say that $u\in SBV(\Om)$ is an {\it almost-quasi minimizer} for the Mumford-Shah functional if there exist $0<\Lambda_1\le\Lambda_2$, $\alpha>0$ and $c_\alpha\ge 0$ such that for every $B _\rho (x) \subset \Omega$ and for every $v \in SBV (\Omega)$ such that $ \{ v  \neq u  \} \subseteq B _\rho (x)$ we have
$$
\int _{B _\rho (x)} |\nabla u| ^ 2 \, d x + \Lambda_1 \mathcal H ^ {d-1} ( J _{u} \cap \overline B _\rho (x))   \leq  
\int _{B _\rho (x)} |\nabla v| ^ 2 \, d x + \Lambda_2  \mathcal H ^ {d-1} ( J _{v} \cap \overline B _\rho (x))  + c _\alpha \rho ^ {d-1 + \alpha}\,. 
$$
In the case $\Lambda_1=\Lambda_2$, such a notion reduces to that of {\it quasi-minimizer} introduced by \textsc{De Giorgi, Carriero} and \textsc{Leaci} in \cite{DGCL}. Under the strict inequality sign, the notion has been introduced in \cite{BL14}, where it was already applied to analyse a (one phase) free discontinuity problem with Robin boundary conditions.
\par
In our analysis, we will employ the following property \cite[Theorem 3.1]{BL14}: the jump set of an almost-quasi minimizers is essentially closed, {\it i.e.},
\begin{equation}
\label{eq:ess-closed-aqm}
\mathcal H^{d-1}\left( (\overline{J_u}\setminus J_u)\cap \Om \right)=0.
\end{equation}

 \section {Proof of Theorem \ref{t:robin} }\label{sec:proof}

\subsection{Existence in $\mathcal F (D)$}

 \begin{proposition}
 Problem $(\overline P)$ admits a solution. 
 \end{proposition} 
 \begin{proof} 
 Let $u^n= (u^n _1, \dots, u^n _k)$ be a minimizing sequence in  $\mathcal F (D)$. It is not restrictive to assume that
  ${ \int _{\R^d}( u^n_i )^ 2} = 1$ for every  $i 
  \in  \{1, \dots, k \}$ and every $n \in \N$.  By comparing with $(\1 _D, 0, \dots, 0 )$, we get  $$  \sum _{i=1} ^k  \int_{\R ^d} |\nabla u ^n _i | ^ 2  \, dx+ \beta \int_{J _{u ^n _i}} \big [((u^n_i) ^+) ^ 2 + ((u^n_i )^-) ^ 2  \big ] \, d \mathcal H ^ { d-1}  \leq C \,.$$ 
  Then, by \cite[Theorem 3.3]{BG10}, up to (not relabeled) sequences, for every $i = 1, \dots, k$, the sequence $ u^n_i $ converges strongly in $L ^ 2 _{\rm loc} (\R ^ d)$ to a function $u _i \in SBV^ {\frac{1}{2}}  (\R ^d)$,  with
 $$
 \int_{\R ^d} |\nabla u_i | ^ 2 \, dx  \leq \liminf _{n \to + \infty}  \int_{\R ^d} |\nabla u^n_i | ^ 2   \, dx 
 $$
 and
 $$
 \int_{J _{u _i}} \big [(u_i ^+) ^ 2 + (u_i ^-) ^ 2  \big ] \, d \mathcal H ^ { d-1}  \leq \liminf _{n \to + \infty}
 \int_{J _{u^n _i}} \big [((u^n_i )^+) ^ 2 + ((u^n_i )^- )^ 2  \big ] \, d \mathcal H ^ { d-1}. 
 $$ 
Notice that  the strong convergence of $u ^ n _i$ to $u _i$ in $L ^ 2 _{\rm loc} (\R ^ d)$ ensures that $u$ still satisfies the conditions ${\rm supp} (u_i) \subseteq \overline D$, $u _i \geq 0$, $\int_{D}u_i^2\,dx=1$, and  $u _i \cdot u _j = 0$  in  $D$.    Summing  the above inequalities over $i = 1, \dots, k$, we see that $u= ( u _1, \dots, u _k)$ is a solution to $(\overline P ) $. 
\end{proof}
 
\subsection{Upper and lower bounds on the supports.} 

   \begin{proposition}\label{p:bounds} Let $u= (u _1, \dots, u_k)\in \mathcal F (D) $ be a solution to problem $(\overline P)$. For every $i = 1, \dots, k$, there exist constants $M _i, \alpha _i$ such that
   \begin{equation}\label{f:Bounds}
   M _i \geq u _i \geq \alpha _ i >0 \qquad \text{a.e.\ on \ } supp(u_i)\,.
   \end{equation}
     \end{proposition} 
\begin{proof} 
Let us derive the lower and upper bounds separately.
\vskip10pt\noindent
{\it Lower bound.} Let $i$ be a fixed index in $\{1, \dots, k \}$. In order to obtain the lower bound in \eqref{f:Bounds}, it is enough to show that there exists $\eta_0$ sufficiently small such that \begin{equation}\label{f:eta} 
\Big | \Big \{ \frac{2}{3} \eta  < u _i < \frac{5}{6} \eta  \Big \} \Big | = 0
\qquad \forall \eta < \eta _0 \,.
\end{equation}
To that aim, we implement the same iteration scheme exploited in the proof of \cite[Theorem 3.5]{BG16}. 
As noticed in \cite[Remark 3.7]{BG16}, such scheme successfully leads to \eqref{f:eta} as soon as one is able to prove the following key estimate 
\begin{equation}\label{f:stimachiave}
E (\e) + c _ 1 \d ^ 2 \gamma (\d, \e) \leq c _ 2 \e ^ 2 h ( \e) \qquad \text{ for a.e. } 0< \d < \e \leq \e _0\,,
\end{equation} 
where
$$\begin{array}{ll} 
& E(\e):= \int _{\{u _i < \e\} } |\nabla u_i| ^ 2 \, dx 
\\
\noalign{\bigskip} 
& \gamma (\d, \e):= \mathcal H ^ { d-1} ( \partial ^ e \{\d < u _i < \e \}  \cap J _{u_i}) 
\\ \noalign{\bigskip}
& h (\e ):= \mathcal H ^ {d-1} (\partial ^ e  \{ u_i \geq \e \} \setminus J _{u_i} ) \, .  
\end{array}
$$ 
Thus, we  limit ourselves to show that \eqref{f:stimachiave} is fulfilled: once this is gained,
to obtain \eqref{f:eta} one can follow exactly the proof of \cite[Theorem 3.5]{BG16}. 

In order to prove \eqref{f:stimachiave}, we compare $u = (u _1, \dots, u _k)$ with $v = (v_1, \dots , v _k)$ defined by 
$$v_j = u _ j \text{  if } j \neq i \quad \text{ and } \quad v _i = u _i \1 _{\{u _i \geq  \e \}}\,.$$ 
Assuming with no loss of generality that $\int _{\R ^ d } u _i ^ 2 = 1$, by the optimality of $u$ we obtain
$$   \begin{array}{ll}
&  \displaystyle
 \int_{\R ^d} |\nabla u_i | ^ 2  \, dx+ \beta \int_{J _{u _i}} \big [(u_i ^+) ^ 2 + (u_i ^-) ^ 2  \big ] \, d \mathcal H ^ { d-1} 
 \leq  \frac{1}{1-\int _{\{u _i <  \e \} }u_i ^ 2 \, dx } \cdot 
 \\ \noalign{\bigskip} 
& \displaystyle 
\left[ \int_{\{u _i \geq  \e \}} |\nabla u_i | ^ 2  \, dx+ \beta \int_{J _{u _i} \cap \{\e \leq u _i ^ - < u _i ^+ \} } \big [(u_i ^+) ^ 2 + (u_i ^-) ^ 2  \big ] \, d \mathcal H ^ { d-1} \right.
\\ \noalign{\bigskip} 
& \displaystyle 
\left.
+ \beta \int_{J _{u _i} \cap \{u _i ^ - < \e \leq u _i ^+ \} } (u_i ^+) ^ 2   \, d \mathcal H ^ { d-1}   + \beta \e ^ 2 \mathcal H ^ {d-1} 
\big ( \partial ^ e \{ u _ i \geq \e \} \setminus J _{u _i}  \big )
\right]\,.
\end{array}
$$ 
Using the inequality $\frac{1}{1- \d} \leq 1 + 2 \d$ holding for $\d \in [0, \frac{1}{2}]$, we deduce that there exists $\e _0 >0$ such that, for $\e < \e _0$, 
$$   \begin{array}{ll}
&  \displaystyle
 \int_{\R ^d} |\nabla u_i | ^ 2  \, dx+ \beta \int_{J _{u _i}} \big [(u_i ^+) ^ 2 + (u_i ^-) ^ 2  \big ] \, d \mathcal H ^ { d-1} 
 \leq \left[ 1 +2 \int _{\{u _i <  \e \} }u_i ^ 2 \, dx \right] \cdot 
 \\ \noalign{\bigskip} 
& \displaystyle 
\left[ \int_{\{u _i \geq  \e \}} |\nabla u_i | ^ 2  \, dx+ \beta \int_{J _{u _i} \cap \{\e \leq u _i ^ - < u _i ^+ \} } \big [(u_i ^+) ^ 2 + (u_i ^-) ^ 2  \big ] \, d \mathcal H ^ { d-1} \right.
\\ \noalign{\bigskip} 
& \displaystyle 
\left.+ \beta \int_{J _{u _i} \cap \{u _i ^ - < \e \leq u _i ^+ \} } (u_i ^+) ^ 2   \, d \mathcal H ^ { d-1}   + \beta \e ^ 2 \mathcal H ^ {d-1} 
\big ( \partial ^ e \{ u _ i \geq \e \} \setminus J _{u _i}  \big )
\right]\,.
\end{array}
$$ 
In turn, since 
$$ \begin{array}{ll} & \displaystyle \int_{\{u _i \geq  \e \}} |\nabla u_i | ^ 2  \, dx+ \beta \int_{J _{u _i} \cap \{\e \leq u _i ^ - < u _i ^+ \} } \big [(u_i ^+) ^ 2 + (u_i ^-) ^ 2  \big ] \, d \mathcal H ^ { d-1}  
 \\  \noalign{\bigskip} 
& \displaystyle +  \beta \int_{J _{u _i} \cap \{u _i ^ - < \e \leq u _i ^+ \} } (u_i ^+) ^ 2   \, d \mathcal H ^ { d-1}   \leq \int_{\R^d} |\nabla u_i | ^ 2  \, dx+ \beta \int_{J _{u _i}  } \big [(u_i ^+) ^ 2 + (u_i ^-) ^ 2  \big ] \, d \mathcal H ^ { d-1}, \end{array}
$$ 
it follows that, for a positive constant $C$ (depending on $u$ but independent of $\e$), 
$$   \begin{array}{ll}
&  \displaystyle
 \int_{\{u _ i < \e\}} |\nabla u_i | ^ 2  \, dx+ \beta \int_{J _{u _i} \cap \partial ^ e \{ u _i < \e \}} \big [(u_i ^+) ^ 2 + (u_i ^-) ^ 2  \big ] \, d \mathcal H ^ { d-1} 
 \\ \noalign{\bigskip} 
&  \leq 
\displaystyle 
\left[ 1 +2 \int _{\{u _i <  \e \} }u_i ^ 2 \, dx \right] \cdot  \Big [ \beta \e ^ 2 \mathcal H ^ {d-1} 
\big ( \partial ^ e \{ u _ i \geq \e \} \setminus J _{u _i}  \big )
\Big ] + C  \int _{\{u _i <  \e \} }u_i ^ 2 \, dx 
 \\ \noalign{\bigskip} 
& \displaystyle 
\le 3 \beta \e ^ 2 \mathcal H ^ {d-1} 
\big ( \partial ^ e \{ u _ i \geq \e \} \setminus J _{u _i}  \big )
 + C  \int _{\{u _i <  \e \} }u_i ^ 2 \, dx
\,.
\end{array}
$$ 
Next we observe that, thanks to the Faber-Krahn inequality \eqref{f:FK}, 
\begin{multline*}
\int _{\{u _i <  \e \} }u_i ^ 2 \, dx\\
 \leq \frac{1}{\lambda _1( \{u _i <  \e \}^* , \beta) } \left[ 
\int_{\{u _ i < \e\}} |\nabla u_i | ^ 2  \, dx+ \beta \int_{J _{u _i} \cap \partial ^ e \{ u _i < \e \}} \big [(u_i ^+) ^ 2 + (u_i ^-) ^ 2  \big ] \, d \mathcal H ^ { d-1}  \right],
\end{multline*}
where $\{u _i <  \e \}^*$ denotes a ball having the same volume as $\{u _i <  \e \}$. We can choose $\e$ so small that
$$\frac{C}{\lambda _1( \{u _i <  \e \}^* , \beta) } \leq \frac{1}{2}\,.$$ 
Therefore, up to reducing $\e _0$, for  $0< \d < \e < \e _0$ it holds 
$$    \int_{\{u _ i < \e\}} |\nabla u_i | ^ 2  \, dx+ \beta \d ^ 2 \mathcal H ^ {d-1} \big ( J _{u _i} \cap \partial ^ e \{\d <  u _i < \e \} \big )
 \leq 
6 \beta \e ^ 2 \mathcal H ^ {d-1} 
\big ( \partial ^ e \{ u _ i \geq \e \} \setminus J _{u _i}  \big )\,.
$$
We have thus shown the validity of \eqref{f:stimachiave}.

\medskip\noindent
{\it Upper bound.} 
 Let $i$ be a fixed index in $\{1, \dots, k \}$. Assume by contradiction that $u _i \not \in L ^ \infty (\R ^d)$.  
 For every $M >0$, we compare $u = (u _1, \dots, u _k)$ with $v = (v_1, \dots , v _k) \in \mathcal F (D)$ defined by 
$$v_j = u _ j \text{  if } j \neq i \quad \text{ and } \quad v _i = u _i \wedge M \,.$$  
By the optimality of $u$, we have
\begin{multline*}
\int_{\R ^d} |\nabla u_i | ^ 2  \, dx+ \beta \int_{J _{u _i}} \big [(u_i^+) ^ 2 + (u_i ^-) ^ 2  \big ] \, d \mathcal H ^ { d-1}\\
\le 
\frac{ \int _{\R^d} u_i ^ 2 \, dx } { \int _{\R^d} ( u _i \wedge M )^ 2 \, dx } \cdot \left[   \int_{\{u _i \leq M \}} |\nabla u_i | ^ 2  \, dx+\beta \int_{J _{u _i} \cap \{u _i ^- < u _i ^+ \leq M \}}  \big [(u_i ^+) ^ 2 + (u_i ^-) ^ 2  \big ] \, d \mathcal H ^ { d-1}\right.\\
 \left. + \beta \int_{J _{u _i} \cap \{u _i ^- <  M  < u _i ^+  \} } \big [M ^ 2 + (u_i ^-) ^ 2  \big ] \, d \mathcal H^{d-1} \right].
\end{multline*}
The above inequality leads to a contradiction by exploiting the assumption that $|\{ u _i > M \} | >0$ for every $M$ and arguing as in the proof of \cite[Theorem 6.11]{BG15}. \end{proof} 

\subsection { SBV regularity and finite $\mathcal H ^ {d-1}$ measure of the jump sets}
   \begin{proposition}  If $u= (u _1, \dots, u_k)\in \mathcal F (D) $ is a solution to problem $(\overline P)$, for every $i = 1, \dots, k$ the function $u_i$ satisfies $\mathcal H ^ {d-1} (J _{u _i}) < + \infty$ and belongs to $SBV  (\R ^d)$. 
\end{proposition} 
\begin{proof} Up to a normalization, we can assume without loss of generality that $\int _{ \R ^ d} u _i ^ 2 \, dx = 1$ for every $i = 1, \dots, k$. 
Since
$$ \sum _{i=1} ^k  \int_{\R ^d} |\nabla u_i | ^ 2  \, dx+ \beta \int_{J _{u _i}} \big [(u_i ^+) ^ 2 + (u_i ^-) ^ 2  \big ] \, d \mathcal H ^ { d-1}  < + \infty\,,$$ and $u _i \geq \alpha _i >0$ a.e.\ on ${\rm supp} (u _i)$ for every $i = 1, \dots, k$, we have
$$\beta \alpha _i ^ 2 \mathcal H ^ {d-1} (J _{u _i}) \leq 
\beta \int_{J _{u _i}} \big [(u_i ^+) ^ 2 + (u_i ^-) ^ 2  \big ] \, d \mathcal H ^ { d-1}  < + \infty \qquad \forall i = 1, \dots, k \,.$$   
In order to prove that $u_i \in SBV (\R ^d)$, we consider the sequence $u _i ^ \e := ( u _i ^ 2 + \e ) ^ {1/2}$.  
If $A$ is an open bounded subset of $\R ^d$, we have $u _i ^ \e \in SBV (A)$ and 
$$\begin{array}{ll} & \sup _{0<\e<1}  \displaystyle \left[ \int _A |\nabla u _i ^ \e | ^ 2 \, dx + \mathcal H ^ {d-1} ( J _{u_i ^ \e } \cap A)  + \| u _i ^ \e \| _{ L ^ \infty (A)} \right]  \\  \noalign{\medskip} 
& \leq\displaystyle \left[ \int _A |\nabla u _i | ^ 2 \, dx + \mathcal H ^ {d-1} ( J _{u_i } \cap A)  + (\| u _i \| ^2_{ L ^ \infty (A)} +1) ^ {1/2} \right] < + \infty,
\end{array}$$
where the last inequality follows from the upper bound in Proposition \ref{p:bounds}. 
By  Ambrosio's compactness theorem \cite{Amb}, since $u _i ^ \e \to u _i$ in $L ^1(A)$,  we deduce that $u_i \in SBV (A)$. 
Moreover, the estimate
$$ \begin{array}{ll} |Du _i| (A)&  \displaystyle \leq \int _A |\nabla u _i | \, dx + 2 \|u _i \| _{\infty} \mathcal H ^ {d-1} ( J _{u _i} \cap A) 
\\ \noalign{\medskip}  
& \leq |D| ^ {1/2}\Big (  \int _{\R ^ d}  |\nabla u _i | ^2 \, dx \Big ) ^ {1/2}  + 2 \|u _i \| _{\infty} \mathcal H ^ {d-1} ( J _{u _i} )\end{array} 
 $$ ensures that $|Du _i | (\R ^d) < + \infty$ and therefore $u _i \in SBV ( \R ^d)$. \end{proof}
 
\subsection{Essential closedness of the jump sets} In the following we denote by $\omega_i$ the support of $u_i$. We will refer to $\{\omega_1,\dots,\omega_k\}$ as the phases of our problem.

\begin{lemma}[\bf Density lower bound for the phases]
\label{p:density}   If $u= (u _1, \dots, u_k)\in \mathcal F (D) $ is a solution to problem $(\overline P)$,  for every $i = 1, \dots, k$ there exist a constant $c_i>0$ and a radius $ \rho _i >0$ such that 
the following property holds true: for every $x \in\R^d$ such that $|B_\rho(x)\cap \omega_i|>0$ for every $\rho>0$, we have
\begin{equation}
\label{f:lowerd}
\frac {|\omega_i \cap B _\rho (x)|  }{\rho ^ d} \geq c _i\qquad \text{for every }\rho \in (0, \rho _i).
\end{equation}

\end{lemma} 
\begin{proof}  Let $i \in \{1, \dots, k\}$ be fixed, and let $x \in\R^d$ such that $|B_\rho(x)\cap \omega_i|>0$ for every $\rho>0$. 
We compare $u = (u _1, \dots, u _k)$ with $v = (v_1, \dots , v _k) \in \mathcal F (D)$ defined by 
$$
v_j = u _ j \text{  if } j \neq i \quad \text{ and } \quad v _i = u _i \1_{\R ^ d \setminus B _\rho (x)}  \,.
$$  

We assume without loss of generality that $\int _{\R ^ d} u _i ^ 2 = 1$, and we write for brevity $B _\rho$ in place of $B _\rho (x)$. 
In order to compute the energy of $v$ we apply Theorem 3.84 in \cite{AFP}, and in particular we denote by $(u_{i, \partial B _\rho} ^+) ^ 2$ the outer trace of $u _i ^2$ on $\partial B _\rho$ defined  for
 $\mathcal  H ^ { d-1}$-a.e.\ $x \in \partial B _\rho$ 
 according to \eqref{f:traces}. 
  By the optimality of $u$ we obtain
\begin{multline*}
 \int_{\R ^d} |\nabla u_i | ^ 2  \, dx+ \beta \int_{J _{u _i}} \big [(u_i ^+) ^ 2 + (u_i ^-) ^ 2  \big ] \, d \mathcal H ^ { d-1}\\
\leq\displaystyle \frac{\int_{\R ^d \setminus B _\rho } |\nabla u_i | ^ 2  \, dx+ \beta \int_{J _{u _i} \setminus B _\rho} \big [(u_i ^+) ^ 2 + (u_i ^-) ^ 2  \big ] \, d \mathcal H ^ { d-1} +  \beta \int_{\partial B _\rho } (u_{i, \partial B _\rho} ^+) ^ 2  \, d \mathcal H ^ { d-1} }{ 1- \int _{B _\rho}   u _ i ^ 2\,dx} \,.
\end{multline*}
Using the inequality $\frac{1}{1- \d} \leq 1 + 2 \d$ holding for $\d \in [0, \frac{1}{2}]$, we deduce that there exists $\rho _0 >0$ such that, for $\rho< \rho _0$, 
$$   
\begin{array}{ll}
&  \displaystyle
 \int_{\R ^d} |\nabla u_i | ^ 2  \, dx+ \beta \int_{J _{u _i}} \big [(u_i ^+) ^ 2 + (u_i ^-) ^ 2  \big ] \, d \mathcal H ^ { d-1} 
 \leq \left[ 1 +2 \int _{B _\rho }u_i ^ 2 \, dx \right] \cdot 
 \\ \noalign{\bigskip} 
& \displaystyle 
\cdot\left[\int_{\R ^d \setminus B _\rho } |\nabla u_i | ^ 2  \, dx+ \beta \int_{J _{u _i} \setminus B _\rho} \big [(u_i ^+) ^ 2 + (u_i ^-) ^ 2  \big ] \, d \mathcal H ^ { d-1} +  \beta \int_{\partial B _\rho }(u_{i, \partial B _\rho} ^+) ^ 2  \, d \mathcal H ^ { d-1} 
\right]\,.
\end{array}
$$ 
 In turn, since 
\begin{multline*}
\int_{\R ^d \setminus B _\rho } |\nabla u_i | ^ 2  \, dx+ \beta \int_{J _{u _i} \setminus B _\rho} \big [(u_i ^+) ^ 2 + (u_i ^-) ^ 2  \big ] \, d \mathcal H ^ { d-1} \\
\leq \int_{\R^d} |\nabla u_i | ^ 2  \, dx+ \beta \int_{J _{u _i}  } \big [(u_i ^+) ^ 2 + (u_i ^-) ^ 2  \big ] \, d \mathcal H ^ { d-1},
\end{multline*}
and in view of the upper bound in \eqref{f:Bounds}, it follows that, for a positive constant $C$ independent of $\rho$, 
\begin{multline}
\label{f:upperb}
\displaystyle \int_{B _\rho} |\nabla u_i | ^ 2  \, dx+ \beta \int_{\overline B _\rho \cap J _{u _i}} \big [(u_i ^+) ^ 2 + (u_i ^-) ^ 2  \big ] \, d \mathcal H ^ { d-1} \\
\leq \displaystyle C \int _{B _\rho} u _i ^ 2 \,dx+ 3 \beta   \int_{\partial B _\rho }(u_{i, \partial B _\rho} ^+) ^ 2  \, d \mathcal H ^ { d-1}. 
\end{multline}
Adding to both sides the term  $\beta\int_{\partial B _\rho }(u_{i, \partial B _\rho} ^-) ^ 2  \, d \mathcal H ^ { d-1}$, where $u_{i, \partial B _\rho} ^-$ denotes the inner trace on $\partial B_\rho$, thanks to the Faber-Krahn inequality \eqref{f:FK} we obtain the estimate
\begin{multline}
\label{f:f1}
\lambda _ 1 (\{\omega _i \cap B _\rho\}^*,\beta)   \int_{B _\rho } u _ i ^ 2\,dx\\
\le\int_{B _\rho} |\nabla u_i | ^ 2  \, dx+ \beta \int_{B _\rho \cap J _{u _i}} \big [(u_i ^+) ^ 2 + (u_i ^-) ^ 2  \big ] \, d \mathcal H ^ { d-1}+
\beta \int_{\partial B _\rho }(u_{i, \partial B _\rho} ^-) ^ 2  \, d \mathcal H ^ { d-1} \\
\le C \int _{B _\rho} u _i ^ 2 \, dx + 3 \beta   \int_{\partial B _\rho } (u_{i, \partial B _\rho} ^-) ^ 2  +(u_{i, \partial B _\rho} ^+) ^ 2  \, d \mathcal H ^ { d-1},
\end{multline}
where $\{\omega _i \cap B _\rho\}^*$ denotes a ball with the same volume as $\omega _i \cap B _\rho$, that is with radius
\begin{equation}\label{f:f2}  
r _i:= \frac{|\omega _i \cap B _\rho| ^ {1/d}}{\omega _d} \,.
\end{equation}
Notice that there exists a positive constant $C'$ (independent of $x$) such that 
\begin{equation}\label{f:f3}
 \lambda _1 (B _{r_i}, \beta) \sim \frac{C'}{|\omega _i \cap B _\rho| ^ {1/d} } \qquad \text  { as } \rho \to 0 ^+ \,.
 \end{equation}  
Indeed, we have 
$$ 
\lambda _1 (B _{r_i}, \beta)  =  \lambda _1 (r_i B _1, \beta)  = \frac{1}{r_i^2} \lambda _ 1 (B_1, \beta r_i  ) = \frac{\beta}{r_i} \frac{1 }{ \beta r_i} \lambda _ 1 (B_1, \beta r _i) \sim \frac{\beta}{r_i}  \frac{|\partial B _1|}{|B_1|}.
$$ 
For the latter asymptotic equivalence, see \cite[Proposition 9]{bf17R} and references therein.
\par
Recalling the upper and lower bounds in \eqref{f:Bounds}, we deduce from \eqref{f:f1}, \eqref{f:f2}, and \eqref{f:f3} that there exists a positive constant $C''$ (independent of $x$) such that  the function  $\theta _i (\rho):=  \big |\omega _i \cap B _\rho \big | $ satisfies,  for $\rho$ sufficiently small, the differential inequality 
$$\theta _i ^ {1 - \frac{1}{d}} (\rho)= \big |\omega _i \cap B _\rho \big | ^ {1- \frac{1}{d}} \leq C''  \big |\omega _i \cap \partial B _\rho \big |  =   C'' \theta _i ' (\rho) \,.$$
 Hence $\theta _i (\rho) \geq C'' \rho ^ d$ for $\rho$ sufficiently small, and the proof of \eqref{f:lowerd} is achieved.  
 \end{proof}

\begin{corollary}
\label{cor:phases}
If $u= (u _1, \dots, u_k)\in \mathcal F (D) $ is a solution to problem $(\overline P)$,  for every $i = 1, \dots, k$ the following items hold true.
\begin{itemize}
\item[(a)] $\partial^e\omega_i\subseteq J_{u_i}$.
\item[(b)] $\omega_i^{(0)}$ is open.
\item[(c)] $C_i:=\R^d\setminus \omega_i^{(0)}$ is closed with  $J_{u_i}\subseteq C_i$ and such that for every $i\not=j$
\begin{equation}
\label{eq:equalityCi}
C_i\cap C_j=\partial^e\omega_i \cap \partial^e\omega_j=J_{u_i}\cap J_{u_j}.
\end{equation}
\end{itemize} 
\end{corollary}

\begin{proof}
\par\noindent
\begin{itemize}
\item[(a)] The inclusion $\partial^e\omega_i \subseteq J_{u_i}$ comes from the lower bound in \eqref{f:Bounds}.

\item[(b)] Let $x\in \omega_i^{(0)}$, and assume by contradiction that there exists $x_n\not\in \omega_i^{(0)}$ with $x_n\to x$. Then $x_n$ is a point which satisfies Lemma \ref{p:density}. But then by \eqref{f:lowerd} we infer that for every $\rho<\rho_i$
$$
\frac{|\omega_i \cap B_\rho(x)|}{\rho^d}=\lim_n \frac{|\omega_i \cap B_\rho(x_n)|}{\rho^d}\ge c_i,
$$
so that $x$ cannot have zero density with respect to $\omega_i$.

\item[(c)] Clearly $J_{u_i}\subseteq C_i$ since jump points have positive density for $\omega_i$. Moroever,
if $x\in C_i\cap C_j$, we have that $x$ has positive density with respect to both $\omega_i$ and $\omega_j$ so that $x\in \partial^e\omega_i \cap \partial^e\omega_j \subseteq J_{u_i}\cap J_{u_j}$. We thus obtain equality \eqref{eq:equalityCi}. 
\end{itemize}
\par
\par
\end{proof}

\begin{lemma}[\bf Almost-quasi minimality]
\label{lem:AQM}
Let $u= (u _1, \dots, u_k)\in \mathcal F (D) $ be a solution to problem $(\overline P)$. Then for every $i = 1, \dots, k$,  the function $u_i$ is an almost quasi minimizer for the Mumford-Shah functional (see Section \ref{sec:mon}) in the open set $D\cap \bigcap_{j\not=i} \omega_j^{(0)}$.
\end{lemma} 

\begin{proof} 
Let us set for brevity
$$
\widehat \omega_i:=\bigcap_{j\not=i} \omega_j^{(0)}.
$$
Let us show that there exists a constant $c>0$ such that, for any $y \in \widehat \omega _i \cap D$ and any $v_i \in \mathcal SBV (\widehat \omega _i\cap D) $ with $ \{ v _i \neq u _i\} \subseteq B _\rho (y) \subset \widehat \omega _i \cap D$, there holds
\begin{multline}
\label{f:cla}
 \int _{B _\rho (y)} |\nabla u _i |^ 2 \, dx + \beta \alpha _i ^ 2 \mathcal H ^ {d-1} \big ( J _{u _i} \cap \overline B _\rho (y) \big )\\
 \le
 \int _{B _\rho (y)}  |\nabla v _i |^ 2 \, dx + 2 \beta \| u _i \| _\infty ^ 2 \mathcal H ^ {d-1} \big ( J _{v _i} \cap \overline B _\rho (y)\big ) + c \rho ^ d,
\end{multline}
where $\alpha_i$ is given in \eqref{f:Bounds}.
 To that aim we can assume without loss of generality that 
\begin{multline}
\label{f:maj}
\int _{B _\rho (y)}  |\nabla v _i |^ 2 \, dx + 2 \beta \| u _i \| _\infty ^ 2 \mathcal H ^ {d-1} \big ( J _{v _i} \cap \overline B _\rho (y) \big )  \\
\le  \int _{B _\rho (y)} |\nabla u _i |^ 2 \, dx + \beta \alpha _i ^ 2 \mathcal H ^ {d-1} \big ( J _{u _i} \cap \overline B _\rho (y) \big ). 
\end{multline}
We consider the function $\tilde v = ( \tilde v _1, \dots, \tilde v _k)$ defined on $D$ by 
  $$\tilde v _j:= 
  \begin{cases}
  u _j & \text { if } j \neq i
  \\  \noalign{\smallskip}
u_i 1_{\R^d\setminus B_\rho(y)}+[|v_i| \wedge \|u_i\|_\infty] \1_{B_\rho(y)}
  & \text{ if } j = i\,.
  \end{cases}$$
Since  $ \{ v _i \neq u _i\}\subseteq B_\rho(y) \subset \widehat \omega _i \cap D$, 
the function $\tilde v$ defines an element of $\mathcal F (D)$. Let us take it as a competitor in problem $(\overline P)$. 
  Exploiting the optimality of $u$ for problem $(\overline P)$, the fact that $\tilde v _j = u _j$ for every $j \neq i$, and and recalling that  $ \{ v _i \neq u _i\} \subset B _\rho (y)$, we obtain (assuming without restriction that $\int_{\R^d}u_i^2\,dx=1$)
\begin{multline*}
 \int_{\R ^d} |\nabla u_i | ^ 2  \, dx+ \beta \int_{J _{u _i}} \big [(u_i ^+) ^ 2 + (u_i ^-) ^ 2  \big ] \, d \mathcal H ^ { d-1} \\
 \le
 \frac{  \int_{\R ^d} |\nabla \tilde v _i | ^ 2  \, dx+ \beta \int_{J _{\tilde v _i}} \big [(\tilde v_i ^+) ^ 2 + (\tilde v_i ^-) ^ 2  \big ] \, d \mathcal H ^ { d-1} }{1- \int _{B _\rho (y)}( u _i ^ 2- \tilde v _i ^ 2) \, dx}\,.
\end{multline*}
Then, by using the estimate
$$ \Big | \int _{B _\rho (y)}( u _i ^ 2- \tilde v _i ^ 2) \, dx \Big | \leq 2 \| u _i \| ^2_\infty \omega _d \rho ^ d \, , $$ 
  and the  inequality $\frac{1}{1- \d} \leq 1 + 2 \d$ holding for $\d \in [0, \frac{1}{2}]$, we obtain, for $\rho$ sufficiently small, 
\begin{multline}
\label{f:sti1}
\int_{\R ^d} |\nabla u_i | ^ 2  \, dx+ \beta \int_{J _{u _i}} \big [(u_i ^+) ^ 2 + (u_i ^-) ^ 2  \big ] \, d \mathcal H ^ { d-1}  \\
\le
\left[  \int_{\R ^d} |\nabla \tilde v _i | ^ 2  \, dx+ \beta \int_{J _{\tilde v _i}} \big [(\tilde v_i ^+) ^ 2 + (\tilde v_i ^-) ^ 2  \big ] \, d \mathcal H ^ { d-1} \right] \cdot \Big [ 1 + 2 \| u _i \| ^2_\infty \omega _d \rho ^ d \Big ].
\end{multline}
  In view of \eqref{f:maj} and of the definition of $\tilde v _i$, we see that
  \begin{equation}\label{f:sti2} \int_{\R ^d} |\nabla \tilde v _i | ^ 2  \, dx+ \beta \int_{J _{\tilde v _i}} \big [(\tilde v_i ^+) ^ 2 + (\tilde v_i ^-) ^ 2  \big ] \, d \mathcal H ^ { d-1} \leq C' \, ,\end{equation}
  for some positive constant $C'$ depending on $u$. 
  
  From \eqref{f:sti1} and \eqref{f:sti2}
   we deduce that, provided $\rho$ is sufficiently small, the required inequality \eqref{f:cla} is satisfied for some positive constant $c$ (depending on $u$).  Since the left hand side of such inequality is bounded in $\rho$, up to increasing $c$ we infer that it continues to hold for every $\rho >0$ such that $B _\rho (y) \subseteq \widehat \omega _i \cap D$. The proof is thus concluded.
\end{proof}
 
\medskip

\begin{proposition}[\bf Essential closedness of the jump sets]
\label{prop:ess-closed}
  If $u= (u _1, \dots, u_k)\in \mathcal F (D) $ is a solution to problem $(\overline P)$, for every $i = 1, \dots, k$ the set $J _{u_i}$ is essentially closed in $D$, {\it i.e.},
\begin{equation}
\label{eq:Juclosed}
\mathcal H^{d-1}\left( \left(\overline{J_{u_i}} \setminus J_{u_i}\right)\cap D\right)=0.
\end{equation}
\end{proposition}

\begin{proof}
The relation
\begin{equation}
\label{eq:hd1}
\mathcal H^{d-1}\left( (\overline{J_{u_i}}\setminus J_{u_i})\cap D \cap \bigcap_{j\not=i} \omega _j^{(0)}\right)=0
\end{equation}
is a consequence of the almost-quasi minimality of Lemma \ref{lem:AQM} (see Section \ref{sec:mon}).
On the other hand, recalling point (c) of Corollary \ref{cor:phases}
$$
\overline{J_{u_i}}\setminus \bigcap_{j\not=i} \omega _j^{(0)}=\overline{J_{u_i}}\cap \bigcup_{j\not=i}C_j\subseteq C_i\cap \bigcup_{j\not=i}C_j\subseteq J_{u_i}
$$
which entails
\begin{equation}
\label{eq:hd2}
(\overline{J_{u_i}}\setminus J_{u_i})\setminus \bigcap_{j\not=i} \omega _j^{(0)}=\emptyset.
\end{equation}
The conclusion follows gathering \eqref{eq:hd1} and \eqref{eq:hd2}.
\end{proof}

\begin{proposition}[\bf Ahlfors regularity]
\label{p:Ahlf} 
  If $u= (u _1, \dots, u_k)\in \mathcal F (D) $ is a solution to problem $(\overline P)$, for every $i = 1, \dots, k$
the set $J_{u_i}$ is Ahlfors regularity in $D$, that is 
  there exist a constant $k _i>0$ and a radius $ r _i >0$ such that,  for  every $x\in J _{u_i}$ and every  $\rho \in (0, r_i]$ such that $B _\rho (x) \subset D$, there holds  
         \begin{equation}\label{f:Ahlf}
   k_i \rho ^ {d-1} \leq \mathcal H ^ {d-1} (J _{u _i} \cap B _\rho (x)) \leq \frac{1}{k_i} \rho ^ {d-1}\,. 
     \end{equation} 
\end{proposition}

 \proof In order to prove the upper bound inequality in \eqref{f:Ahlf},   we proceed as in the proof of Lemma \ref{p:density} until we arrive at inequality \eqref{f:upperb}. Using such inequality and  Proposition \ref{p:bounds}, we obtain 
$$
\begin{array}{ll} 
 \beta  \alpha _i ^ 2 \mathcal H ^ {d-1}  (J _{u _i} \cap B _\rho (x))  &\displaystyle  \leq 
\int_{B _\rho(x)} |\nabla u_i | ^ 2  \, dx+ \beta \int_{B _\rho(x) \cap J _{u _i}} \big [(u_i ^+) ^ 2 + (u_i ^-) ^ 2  \big ] \, d \mathcal H ^ { d-1} \\
 \noalign{\medskip} &\displaystyle \leq 
 C \int _{B _\rho} u _i ^ 2 + 3 \beta   \int_{\partial B _\rho(x) }(u_{i, \partial B _\rho} ^+) ^ 2  \, d \mathcal H ^ { d-1}  
 \\
 \noalign{\medskip} &\displaystyle \leq 
 C M _i ^ 2 \omega _d \rho ^ d  + 3 \beta M _ i ^ 2 d\omega _{d} \rho ^ {d-1}\,,
  \end{array}
  $$
which clearly implies the validity of the upper bound inequality in \eqref{f:Ahlf} for $\rho$ sufficiently small. 
\par

Concerning the lower bound, let us employ the following notation:
$$
\widehat \omega_i:=\bigcap_{j\not=i} \omega_j^{(0)},\qquad \delta_i(x):=d(x,\partial \widehat \omega_i)\,. \EEE 
$$
Let us fix $x\in D$, and let us distinguish the two cases 
$$
x\in J_{u_i}\setminus \widehat \omega_i\qquad\text{and}\qquad x\in J_{u_i}\cap \widehat \omega_i.
$$
\vskip10pt\noindent{\bf Case 1.} Let $x\in J_{u_i}\setminus \widehat \omega_i$. By Corollary \ref{cor:phases}, we have $x\in C_i\cap C_j$ for some $j\not=i$. In view of the relative isoperimetric inequality
  \begin{equation}\label{f:rel-iso}
  \Big (\min \Big \{ \big  |B _\rho (x) \cap \omega _i \big |\, ,  \, \big |B _\rho (x) \setminus \omega _i \big | \Big \}  \Big ) ^ {\frac{d-1}{d}} \leq c _d \mathcal H ^ {d-1} \big (\partial ^ e \omega _i \cap B _\rho (x) \big ) \,,
  \end{equation} 
together with the inequality $|B _\rho (x) \setminus \omega _i \big | \supseteq  |B _\rho (x) \cap \omega _j\big |$ and the density lower bound \eqref{f:lowerd} for $\omega_i$ and $\omega_j$ at the point $x$, we infer
that there exist $\rho_i'>0$ and $k_i'>0$ (independent of $x$) such that for every $\rho<\rho_i'$
$$
\mathcal H^{d-1}(J_{u_i}\cap B_\rho(x))\ge \mathcal H^{d-1}(\partial^e\omega_i\cap B_\rho(x))\ge k'_i \rho^{d-1}.
$$

\vskip10pt\noindent{\bf Case 2.} Let $x\in J_{u_i}\cap \widehat \omega_i$. Recall that by Lemma \ref{lem:AQM} the function $u_i$ is an almost-quasi minimizer for the Mumford-Shah functional in the open set $D\cap  \widehat \omega_i$, which entails the Ahlfors regularity of its jump set (see \cite[Section 3.2]{BL14}). 
As a consequence, there exist a radius $\rho_i''>0$ and a constant $k_i''>0$ (independent of $x$) such that for every $\rho<\rho''_i\wedge \delta_i(x)$ 
\begin{equation}
\label{eq:ahl-ju-1}
\mathcal H^{d-1}(J_{u_i}\cap B_\rho(x))\ge k''_i \rho^{d-1}.
\end{equation}
Let us extend this lower bound on balls $B_\rho(x)$ contained in $D$ (and not only in $D\cap \widehat \omega_i$).
\par
To that aim we first remark that, up to changing $k''_i$ into $k''_i/m ^ {d-1}$, 
   the validity
   of \eqref{eq:ahl-ju-1} can be extended to radii $\rho \in (0, m(\rho''_i \wedge \delta _i (x)))$ for any $m \in \N$. Indeed by applying \eqref{eq:ahl-ju-1} with $\frac{\rho}{m}$ in place of $\rho$ we get:
   \begin{equation}\label{f:sit2}
    \mathcal H ^ {d-1} (J _{u_i} \cap B _\rho (x)) \geq     \mathcal H ^ {d-1} (J _{u_i} \cap B _{\frac{\rho}{m}} (x)) \geq \frac{k_i '}{m ^ {d-1}} \rho ^{d-1}
\end{equation}
for every $\rho \in (0, m(\rho''_i \wedge \delta _i (x)))$. Since we can choose $m$ large enough (independent of $x$) so that 
\begin{equation}\label{f:stimam}
m(\rho''_i \wedge \delta _i (x))  \geq 2 \d_ i (x)\,, 
\end{equation} 
we are reduced to show the lower bound inequality in \eqref{f:Ahlf} for the radii $\rho \geq 2 \d_i (x)$ such that $B _\rho (x) \subset D$.  For such a radius, we can proceed in a similar way as done for points $x \in J _{u _i} \setminus \widehat \omega _i$, namely we prove the inequality with $J _{u_i}$ replaced by $\partial ^ e \omega _i$. 
\par
To such purpose, let $x _j \in C _j$ be a point such that $|x_j - x |= \delta_i (x)$. Then we may write thanks to Lemma \ref{p:density}
$$
|B _\rho (x) \setminus \omega _i|
  \geq |\omega _j \cap B _\rho (x)  | \geq |  \omega _j \cap B _{\rho - \delta_i(x)} ( x_j) | \geq c _j (\rho - \delta_i(x)) ^ d  \geq \frac{c_j} {2^d} \rho^d. \,    
$$ 
Thanks to the isoperimetric inequality \eqref{f:rel-iso}, we infer that, for any $\rho \geq 2 \d_i (x)$, $\mathcal H ^ {d-1} \big (\partial ^ e \omega _i \cap B _\rho (x) \big ) $, and hence $\mathcal H ^ {d-1} \big (J_{u _i} \cap B _\rho (x) \big ) $ is bounded from below by a constant (independent of $x$) times $\rho ^ {d-1}$. Combining this assertion with \eqref{f:sit2}-\eqref{f:stimam}, we have achieved the proof of the lower bound inequality in \eqref{f:Ahlf} also for points $x \in J _{u _i} \cap \widehat \omega _i$.

\qed

\subsection{Identification of an optimal $k$-tuple in $\mathcal A (D)$ and conclusion.} 
  
\begin{proposition}
\label{p:supports}  
If $u= (u _1, \dots, u_k)\in \mathcal F (D) $ is a solution to problem $(\overline P)$, there exists a $k$-tuple of open connected sets $(\Omega _1, \dots, \Omega _k) \in \mathcal A ( D)$ such that, for every $i = 1, \dots, k$, 
$$
\partial \Omega _i= \overline { J _{u_i}}, 
\qquad 
\mathcal H ^ {d-1}\left( (\partial \Omega _i \setminus J _{u _i} )\cap D\right)= 0, 
\qquad u _i = 0 \text{ a.e. on } \R ^ d \setminus \Omega _i \,.
$$
Moreover, for every $i = 1, \dots, k$, the function $u _i$ belongs to $H ^ 1 (\Omega _i ) \cap L ^ \infty (\Omega _i)$ and satisfies $u _i \geq \alpha _i >0$ a.e.\ on $\Omega _i$. 
\end{proposition}
  
  \begin{proof} We define $\Omega _i$ as the union of the connected components of $\R^d\setminus \overline { J _{u _i}}$ where $u _i$ is not identically zero. Clearly, by construction, $\Omega _i\subseteq D$ is open, it satisfies 
$\partial \Omega _i= \overline { J _{u_i}}$,  $u _i = 0$ a.e.\ on $\R ^ d \setminus \Omega _i$, and $\Omega _i \cap \Omega _j = \emptyset$ for $j \neq i$ (the latter condition comes from $u _i \cdot u _j = 0$ for $j \neq i$). 
\par
 The property that $(\partial \Omega _i \setminus J _{u _i})\cap D$ is $\mathcal H ^ {d-1}$-negligible follows from Proposition \ref{prop:ess-closed}. Moreover, since
 $$
 \partial \Om_i\subseteq \partial D\cup ( \overline { J _{u_i}} \cap D),
 $$
 we infer $\Omega_i\in \mathcal A(D)$.
 \par
 The fact that $\Omega _i$ is connected can be easily proved by contradiction. Namely, if $\Omega _i = \Omega _i ' \cup \Omega _i ''$, with $\Omega _i '$ and $\Omega _i ''$ nonempty disjoint open sets, letting $u _i '  := u _i \cdot \1_{\Omega _i ' }$ and  
  $u _i ''  := u _i \cdot \1_{\Omega _i '' }$, considering one of the two $k$-tuples   $(u _1, \dots, u _i ', \dots , u _k )$ and $(u _1, \dots, u _i '', \dots , u _k ) $,
  and arguing as in the proof of \cite[Theorem 6.15]{BG15} leads to 
contradict the optimality of $(u _1, \dots, u _i, \dots, u _k)$ for problem $(\overline  P )$. 
\par
Finally, from the definition of $\Omega _i$ we deduce  that $u _i \in H ^ 1 (\Omega _i)$, and the remaining part of the statement follows from  
  Proposition \ref{p:bounds}. 
  \end{proof}
  
\noindent{\bf Conclusion of the proof of Theorem \ref{t:robin}}. Let $u= (u _1, \dots, u_k)\in \mathcal F (D) $ be a solution to problem $(\overline P)$, and let $(\Omega_1, \dots, \Omega _k)$ be a $k$-tuple of open connected sets as in 
  Proposition \ref{p:supports}. Let us show that $(\Omega _1, \dots, \Omega _k)$ solves problem $(P)$, namely, for every $k$-tuple $(A_1, \dots, A_k)\in \mathcal A (D)$, it holds
  \begin{equation}\label{f:opt}
  \sum _{i = 1} ^k \lambda _ 1 (\Omega _i, \beta) \leq  \sum _{i = 1} ^k \lambda _ 1 (A _i, \beta) \,.
  \end{equation}
  First we observe that, since $u _i \in H ^ 1 (\Omega _i ) \cap L ^ \infty (\Omega _i)$ with $u _i \geq \alpha _i >0$ a.e.\ on $\Omega _i$ ({\it cf.} Proposition \ref{p:supports}), $u _i$ is admissible in the minimization problem which defines  $\lambda _1 (\Omega _i, \beta)$ according to \eqref{f:deflambda}.  

Notice that
\begin{multline*}
\int_{\partial \Omega_i \cap \partial D} \big [(u_i^+) ^ 2 + (u_i ^-) ^ 2  \big ] \, d \mathcal H ^ { d-1}  
=
\int_{\overline{J_{u_i}}\cap \partial D} \big [(u_i^+) ^ 2 + (u_i ^-) ^ 2  \big ] \, d \mathcal H ^ { d-1}  
\\
=\int_{J_{u_i} \cap \partial D} \big [(u_i^+) ^ 2 + (u_i ^-) ^ 2  \big ] \, d \mathcal H ^ { d-1}  
\end{multline*}
since $u^\pm_i=0$ $\mathcal H ^ {d-1}$-a.e.\ on $(\overline{J_{u_i}}\setminus J_{u_i})\cap \partial D$ (thanks to the regularity assumed on $D$). We thus may write
\begin{multline*}
\int_{\partial \Omega_i} \big [(u_i^+) ^ 2 + (u_i ^-) ^ 2  \big ] \, d \mathcal H ^ { d-1}\\
=
\int_{\partial \Omega_i \cap \partial D} \big [(u_i^+) ^ 2 + (u_i ^-) ^ 2  \big ] \, d \mathcal H ^ { d-1} +
\int_{\partial \Omega_i \cap D} \big [(u_i^+) ^ 2 + (u_i ^-) ^ 2  \big ] \, d \mathcal H ^ { d-1}  \\
=
\int_{J_{u_i} \cap \partial D} \big [(u_i^+) ^ 2 + (u_i ^-) ^ 2  \big ] \, d \mathcal H ^ { d-1}  +
\int_{\overline{J_{u_i}} \cap D} \big [(u_i^+) ^ 2 + (u_i ^-) ^ 2  \big ] \, d \mathcal H ^ { d-1} \\
=\int_{J_{u_i}} \big [(u_i^+) ^ 2 + (u_i ^-) ^ 2  \big ] \, d \mathcal H ^ { d-1}, 
\end{multline*}
where the last equality follows from \eqref{eq:Juclosed} in Proposition \ref{prop:ess-closed}. We deduce
\begin{multline}
\label{f:stima1}
 \sum _{i = 1} ^k \lambda _ 1 (\Omega _i, \beta) \displaystyle  \leq  \sum _{i = 1} ^k
 \frac{ \int_{\Omega_i} |\nabla u_i | ^ 2  \, dx   + \beta \int_{\partial \Omega_i} \big [(u_i^+) ^ 2 + (u_i ^-) ^ 2  \big ] \, d \mathcal H ^ { d-1} }
 { \int _{\Omega} u_i^ 2 \, dx }  \\
 =  \sum _{i = 1} ^k
 \frac{ \int_{\R^d} |\nabla u_i | ^ 2  \, dx   + \beta \int_{J _{u _i}} \big [(u_i^+) ^ 2 + (u_i ^-) ^ 2  \big ] \, d \mathcal H ^ { d-1} }
 { \int _{\R^d} u_i^ 2 \, dx }.
\end{multline}
Let now $w _i \in (H ^ 1 (A _i ) \cap L ^ \infty (A _i) ) \setminus \{ 0 \}$, and set $\tilde w _i := |w_i| \cdot \1 _{A_i}$.  Then $\tilde w := (\tilde w_1, \dots, \tilde w _k)  \in \mathcal F (D)$ is an admissible competitor for problem $(\overline P)$. We have 
\begin{multline}
\label{f:stima2}
 \sum _{i = 1} ^k
 \frac{ \int_{\R^d} |\nabla u_i | ^ 2  \, dx   + \beta \int_{J _{u _i}} \big [(u_i^+) ^ 2 + (u_i ^-) ^ 2  \big ] \, d \mathcal H ^ { d-1} }
 { \int _{\R^d} u_i^ 2 \, dx }  \\
\le  \displaystyle \sum _{i = 1} ^k
 \frac{ \int_{\R^d} |\nabla \tilde w_i | ^ 2  \, dx   + \beta \int_{J _{\tilde w _i}} \big [(\tilde w_i^+) ^ 2 + (\tilde w_i ^-) ^ 2  \big ] \, d \mathcal H ^ { d-1} }
 { \int _{\R^d} \tilde w _i^ 2 \, dx }\\
\le  \displaystyle \sum _{i = 1} ^k
 \frac{ \int_{A_i} |\nabla  w_i | ^ 2  \, dx   + \beta \int_{\partial A _i} \big [( w_i^+) ^ 2 + (w_i ^-) ^ 2  \big ] \, d \mathcal H ^ { d-1} }
 { \int _{A_i} w _i^ 2 \, dx },
\end{multline}
 where the first inequality follows from  the optimality of $u$ for problem $(\overline P)$, and the second one 
from the  definition of $\tilde w _i$ (which ensures in particular that $\mathcal H ^ {d-1} ( J _{\tilde w _i} \setminus \partial A _i )= 0$).  
  
By combining \eqref{f:stima1} and \eqref{f:stima2} we obtain
$$  \sum _{i = 1} ^k \lambda _ 1 (\Omega _i, \beta)\leq \sum _{i = 1} ^k
 \frac{ \int_{A_i} |\nabla  w_i | ^ 2  \, dx   + \beta \int_{\partial A _i} \big [( w_i^+) ^ 2 + (w_i ^-) ^ 2  \big ] \, d \mathcal H ^ { d-1} }
 { \int _{A_i} w _i^ 2 \, dx },
 $$
 and the required inequality \eqref{f:opt} follows by passing to the infimum over $(w_1, \dots, w_k)$.

  \bigskip
  \noindent {\bf Acknowledgments.} The first author was supported by the  "Geometry and Spectral Optimization"  research programme 
LabEx PERSYVAL-Lab GeoSpec (ANR-11-LABX-0025-01) and ANR Comedic (ANR-15-CE40-0006). 

\end{document}